\documentclass{amsproc}

\newtheorem{theorem}{Theorem}[section]
\newtheorem{lemma}[theorem]{Lemma}
\theoremstyle{definition}
\newtheorem{definition}[theorem]{Definition}
\newtheorem{example}[theorem]{Example}

\theoremstyle{remark}

\begin{document}
\title{Local stability of energy estimates for the Navier--Stokes equations.}

\author{Diego Chamorro}
\address{LAMME, Univ Evry, CNRS, Universit\'e Paris-Saclay, 91025, Evry, France} 
\curraddr{IBGBI, 23 bd de France, 91037 \'Evry cedex, France}
\email{diego.chamorro@univ-evry.fr}
 
\author{Pierre Gilles Lemari\'e-Rieusset}
\address{LAMME, Univ Evry, CNRS, Universit\'e Paris-Saclay, 91025, Evry, France}
\curraddr{IBGBI, 23 bd de France, 91037 \'Evry cedex, France}
\email{pierregilles.lemarierieusset@univ-evry.fr}
 
\author{Kawther Mayoufi}
\address{LAMME, Univ Evry, CNRS, Universit\'e Paris-Saclay, 91025, Evry, France}
\curraddr{IBGBI, 23 bd de France, 91037 \'Evry cedex, France}
\email{kawther.mayoufi@univ-evry.fr}
 
\date{.}


\begin{abstract}
We study the regularity of the weak limit of a sequence of dissipative solutions to the Navier--Stokes equations when no assumptions is made on the behavior of the pressures.
\end{abstract}

\maketitle
%%%%%%%%%%%%%%%%%%%%%%%%%%%%%%%%
\section{Local weak solutions.}\label{LOCA}
In this paper, we are interested in local properties (regularity, local energy estimates) of weak solutions of Navier--Stokes equations. 
   
 \begin{definition}[\bf Local weak solutions] Let $\Omega$ be a   domain in $\mathbb{R}\times\mathbb{R}^3$ and $\vec f\in L^2_{\rm loc}(\Omega)$ a divergence-free time-dependent vector field. A vector field $\vec u$ will be said to be a \emph{local weak solution} of the Navier--Stokes equations on $\Omega$ (associated to the force $\vec f$) if, for each  cylinder $Q=I\times O$ (where $I$ is an open  interval in  $\mathbb{R}$ and $O$ an open subset of $\mathbb{R}^3$) such that $\bar Q$ is a compact subset of $\Omega$, we have $\vec u\in L^\infty_tL^2_x(Q)\cap L^2_t H^1_x(Q)$, $\vec u$ is divergence-free and, for every smooth compactly supported divergence-free vector field $\vec\phi\in \mathcal{D}(Q)$ we have
\begin{equation}\label{WNS} \iint_Q \vec u\cdot(\partial_t\vec\phi+\Delta\vec\phi)+ \vec u\cdot(\vec u\cdot\vec\nabla\vec\phi)+\vec f\cdot\vec\phi
 \, dt\, dx=0.\end{equation}
 \end{definition}
% $\ $\\
More precisely, we shall address the behavior of a weak limit of regular solutions. 

 \begin{definition}[\bf Regular local   solutions] Let $\Omega$ be a domain in $\mathbb{R}\times\mathbb{R}^3$ and $\vec f\in L^2_{\rm loc}(\Omega)$ a divergence-free time-dependent vector field and    $\vec u$    a local weak solution of the Navier--Stokes equations on $\Omega$ (associated to the force $\vec f$).
\begin{itemize}
\item[A)] $\vec u$ is a \emph{regular local solution} if, for each  cylinder $Q\subset\subset \Omega$,   we have $\vec u\in L^\infty_{t,x}(Q) $,  
\item[B)]  The set  $R(\vec u)$ of \emph{regular points} of $\vec u$ is the largest open subset of $\Omega$ on which $\vec u$ is a regular solution. The set $\Sigma(\vec u)$ of \emph{singular points} is the complement of $R(\vec u)$ : $\Sigma(\vec u)=\Omega\setminus R(\vec u)$.
\end{itemize}
 \end{definition}
Our result is then the following one \cite{MAY} :
 
  \begin{theorem}[\bf Singular points of a weak limit.]\label{TheoCLM}
 Let $\Omega$ be a   domain in $\mathbb{R}\times\mathbb{R}^3$. Assume that we have sequences $\vec f_n$ of  divergence-free time-dependent vector fields and    $\vec u_n$  of   local  weak solutions of the Navier--Stokes equations on $\Omega$ (associated to the forces $\vec f_n$) such that, for each  cylinder $Q\subset\subset \Omega$,   we have
 \begin{itemize} 
 \item $\vec f_n\in L^2_t H^1_x(Q)$ and $\vec f_n$ converges weakly in $L^2_{t} H^1_{x}$ to a limit $\vec f$,
 \item the sequence $\vec u_n$ is bounded in $L^\infty_tL^2_x(Q)\cap L^2_t H^1_x(Q)$ and converges weakly in $ L^2_t H^1_x(Q)$ to  a limit $\vec u$ ,
 \item for every $n$, $\vec u_n$ is bounded on $Q$.
 \end{itemize}
 Then the limit $\vec u$ is a local weak solution on $\Omega$ of the Navier--Stokes equations associated to the force $\vec f$, and its set $\Sigma(\vec u)$ has parabolic one-dimensional Hausdorff measure equal to $0$.
\end{theorem}
As we shall see, the main tool of the proof is an extension of the Caffarelli--Kohn--Nirenberg theory \cite{CKN} to the case where we have no control on the pressure (i.e.  the case of generalized suitable solutions \cite{WOL} or dissipative solutions \cite{CLM}).
 
\section{Pressure.}
Equations (\ref{WNS})  can classically be rewritten as an equation involving a pressure term. See for instance \cite{WOL}. In the following, we shall only need the pressure inside spherical cylinders  $Q=I\times B$ (where $I$ is an open  interval in  $\mathbb{R}$ and $B$ an open ball  of $\mathbb{R}^3$). In that case, it is very easy to define a pressure $p$ such that 
\begin{equation} \label{PRES}\partial_t\vec u=\Delta\vec u-\vec u\cdot\vec\nabla\vec u-\vec \nabla p+\vec f \quad \text{ in } \mathcal{D}'(Q). \end{equation} 
Indeed, let $Q$, $Q^\#$, and $ Q^*$ be three relatively compact cylinders in $\Omega$ with $\overline{Q}\subset Q^\#$ and $\overline{Q^\#}\subset  Q^*$  and $\psi$ a cut off smooth function supported in $ Q^*$ and equal to $1$ on a neighboorhood of $\overline{Q^\#}$. The function $$p_0=-\frac{1}{\Delta} \left(\sum_{i=1}^3\sum_{j=1}^3\partial_i\partial_j( \psi u_i u_j)\right)$$
 belongs to $L^2_t L^{3/2}_x$ and, on $Q^\#$, the distribution $$\vec T=\partial_t\vec u-\Delta\vec u+\vec u\cdot \vec\nabla\vec u+\vec \nabla p_0-\vec f$$ satisfies 
 $$ \text{ curl }\vec T=0\text{ and } \text{div }\vec T=0.$$
 Moreover, $\vec T_0=\vec T-\partial_t\vec u$ belongs to $L^2_t H^{-2}_x(Q^\#) $. Picking $t_0\in I$, we define $\vec S=\vec u+\displaystyle{\int_{t_0}^t \vec T_0(s,.)\, ds}$. We have $\vec S\in L^\infty_t H^{-2}_x(Q^\#)$. Moreover, we have $\partial_t  \text{ curl }\vec S=0$ and $\partial_t \text{ div }\vec S=0$. Thus, if $\alpha\in\mathcal{D}(I)$ with $\int\alpha\, dt=1$, we find that 
 $$ \vec S_0=\vec S-\int_I \alpha(s)\vec S(s,.)\, ds$$ satisfies
 $$ \partial_t\vec S_0=\vec T, \quad \text{ curl }\vec S_0=0 \text{ and } \text{ div }\vec S_0=0.$$
 In particular, $$\Delta\vec S_0= \vec\nabla(\text{ div }\vec S_0)-\vec\nabla\wedge(\text{ curl }\vec S_0)=0.$$
 Thus, we get that $\vec S_0$ is smooth in the space variable; in particular $\vec S_0\in L^\infty_t W^{1,\infty}_x(Q)$. If $x_0\in B$ and if we define 
 $$\varpi(t,x)=\int_0^1 \vec S_0(t,(1-\theta)x_0+\theta x)\cdot (x-x_0)\, d\theta,$$ 
 we find that $\varpi\in L^\infty_{t,x}(Q)$ and $\vec\nabla\varpi=\vec S_0$.
 Defining $p=p_0-\partial_t\varpi$, we find the equality (\ref{PRES}).
\\

Of course, the pressure may be singular in time (as $\partial_t\varpi$ is only the derivative of a bounded function). We shall comment further on this in Sections \ref{ENER} and  \ref{SERR}. 

\section{Energy balance.}\label{ENER}

This section is devoted to the study of $\partial_t\vert\vec u\vert^2$, as it is the main tool to estimate the partial regularity of $\vec u$. If $\vec u$ and the pressure $p$ were regular, we could write from equality (\ref{PRES})
$$\partial_t\vert\vec u\vert^2=2 \vec u\cdot \partial_t\vec u=2\vec u\cdot \Delta\vec u-2\vec u\cdot(\vec u\cdot\vec\nabla\vec u+\vec\nabla p)+2\vec u\cdot\vec f$$ and rewrite
$$ 2\vec u\cdot\Delta\vec u=\Delta(\vert \vec u\vert^2)-2\vert\vec\nabla\otimes\vec u\vert^2$$
and, since $\text{ div }\vec u=0$,
$$2\vec u\cdot(\vec u\cdot\vec\nabla\vec u+\vec\nabla p)=\text{ div }((\vert \vec u\vert^2+2p)\vec u).$$
This would give the following local energy balance in $Q$
\begin{equation}\label{ENERa}
\partial_t\vert\vec u\vert^2= \Delta(\vert \vec u\vert^2)-2\vert\vec\nabla\otimes\vec u\vert^2-\text{ div }((\vert \vec u\vert^2+2p)\vec u)+2\vec u\cdot\vec f.\end{equation}
However,   local weak solutions (and their  associates pressures) are not regular enough to allow those computations : the problem lies in the fact that the terms $\vec u\cdot(\vec u\cdot\nabla\vec u)$ and $\vec u\cdot\vec\nabla p$ are not well defined in $\mathcal{D}'$. If the pressure is regular enough (for instance, $p\in L^{3/2}_{t,x}(Q)$)  then one first smoothens $\vec u$ with a mollifier $\varphi_\epsilon=\frac{1}{\epsilon^3}\varphi(\frac x\epsilon)$, defining $\vec u_\epsilon=\varphi_\epsilon*\vec u$. One then finds
$$ \partial_t\vert\vec u_\epsilon\vert^2= \Delta(\vert \vec u_\epsilon\vert^2)-2\vert\vec\nabla\otimes\vec u_\epsilon\vert^2-2 \vec u_\epsilon\cdot\varphi_\epsilon*(\vec u.\vec\nabla\vec u)   -2\text{ div }( (p*\varphi_\epsilon)\vec u_\epsilon)+2\vec u_\epsilon\cdot(\varphi_\epsilon*\vec f). $$ 
The limit $\epsilon\rightarrow 0$ gives then
\begin{equation*}
\partial_t\vert\vec u\vert^2= \Delta(\vert \vec u\vert^2)-2\vert\vec\nabla\otimes\vec u\vert^2- 2\lim_{\epsilon\rightarrow 0}\vec u_\epsilon\cdot\varphi_\epsilon*(\vec u\cdot\vec\nabla\vec u)   -2\text{ div } (p\vec u)+2\vec u\cdot\vec f.\end{equation*}
In order to compare this expression with (\ref{ENERa}), we define
$$ M_\epsilon(\vec u)=- \text{ div }(\vert \vec u\vert^2\vec u) +2 \vec u_\epsilon\cdot\varphi_\epsilon*(\vec u\cdot\vec\nabla\vec u) $$
and write
\begin{equation*}
\partial_t\vert\vec u\vert^2= \Delta(\vert \vec u\vert^2)-2\vert\vec\nabla\otimes\vec u\vert^2-\text{ div }((\vert \vec u\vert^2+2p)\vec u)+2\vec u\cdot\vec f-\lim_{\epsilon\rightarrow 0} M_\epsilon(\vec u).
\end{equation*}
However, our assumptions on weak solutions don't allow us to make all those computations, as the pressure we can define on $Q$ has no regularity with respect to the time variable, so that $p\vec u$ is not well defined in $\mathcal{D}'$. Thus, one must smoothens as well $\vec u$ with respect to the time variable, with a mollifier $\psi_\eta(t)=\frac{1}{\eta}\psi(\frac t \eta)$. Defining $\vec u_{\epsilon,\eta}=\psi_\eta*_t\varphi_\epsilon*_x\vec u=\xi_{\eta,\epsilon}*_{t,x}\vec u$, one finds
\begin{eqnarray*}
\partial_t\vert  \vec u_{\epsilon,\eta}\vert^2&=& \Delta(\vert \vec   u_{\epsilon,\eta}\vert^2)-2\vert\vec\nabla\otimes  \vec u_{\epsilon,\eta}\vert^2-2 \vec u_{\epsilon,\eta}\cdot  \xi_{\eta,\epsilon}*(\vec u\cdot \vec\nabla\vec u) \\
& &  -2\text{ div }( (p*\xi_{\eta,\epsilon})\vec u_{\epsilon,\eta})+2  \vec u_{\epsilon,\eta}\cdot(\xi_{\eta,\epsilon}*\vec f). 
\end{eqnarray*}
The limit $\eta\rightarrow 0$ gives then
$$ \partial_t\vert\vec u_\epsilon\vert^2= \Delta(\vert \vec u_\epsilon\vert^2)-2\vert\vec\nabla\otimes\vec u_\epsilon\vert^2-2 \vec u_\epsilon\cdot\varphi_\epsilon*(\vec u\cdot\vec\nabla\vec u)   -2\lim_{\eta\rightarrow 0}\text{ div }( (p*\xi_{\eta,\epsilon})\vec u_{\epsilon,\eta})+2\vec u_\epsilon\cdot(\varphi_\epsilon*\vec f). $$
The limit $\epsilon\rightarrow 0$ gives finally
\begin{equation}\label{ENERb}
\begin{split}
\partial_t\vert\vec u\vert^2&= \Delta(\vert \vec u\vert^2)-2\vert\vec\nabla\otimes\vec u\vert^2\\
&- 2\lim_{\epsilon\rightarrow 0} \left(\vec u_\epsilon\cdot\varphi_\epsilon*(\vec u\cdot\vec\nabla\vec u)+\lim_{\eta\rightarrow 0}\text{ div }( (p*\xi_{\eta,\epsilon})\vec u_{\epsilon,\eta})\right)+2\vec u\cdot\vec f.
\end{split}
\end{equation}\\

In order to circumvene the problems of lack of regularity for the pressure, we introduce the notion of harmonic correction :
  
 \begin{definition}[\bf Harmonic corrections]   Let $\Omega$ be a   domain in $\mathbb{R}\times\mathbb{R}^3$,  $\vec f\in L^2_{\rm loc}(\Omega)$ a divergence-free time-dependent vector field and    $\vec u$    a local weak solution of the Navier--Stokes equations on $\Omega$ (associated to the force $\vec f$). A \emph{harmonic correction} $\vec H$ on a cylinder $Q\subset\subset\Omega$ is a vector field such that
 \begin{itemize}
 \item $\text{ div } \vec H=0$ and $\Delta\vec H=0$,
 \item $\vec H\in  L^\infty_{t,x}(Q)$ and $\partial_i\vec H\in  L^\infty_{t,x}(Q)$ for $i=1,2,3$,
 \item there exists $\vec F\in L^2_{t,x}(Q)$ and $P\in L^{3/2}_{t,x}(Q)$ such that the vector field $\vec U=\vec u+\vec H$ satisfies 
 \begin{equation*} 
 \partial_t\vec U=\Delta \vec U-\vec U\cdot\vec\nabla\vec U-\vec\nabla P+\vec F.
 \end{equation*}
 \end{itemize}
 \end{definition}
 In the literature, one can find at least two such harmonic corrections for local weak solutions :
 \begin{lemma}  \label{CORREC}
  Let $\Omega$ be a   domain in $\mathbb{R}\times\mathbb{R}^3$,  $\vec f\in L^2_{\rm loc}(\Omega)$ a divergence-free time-dependent vector field and    $\vec u$    a local weak solution of the Navier--Stokes equations on $\Omega$ (associated to the force $\vec f$). Let $Q$ be a spherical cylinder in $\Omega$. Then:
\begin{itemize}
\item[A)]the decomposition of the pressure $p$ as $p=p_0-\partial_t\varpi$ described in Section \ref{LOCA} provides a harmonic correction $\vec H=-\vec\nabla\varpi$ of $\vec u$ on $Q$,
\item[B)]Let $\psi(t,x)=\alpha(t)\beta(x)$ be a smooth cut-off function supported by a   cylinder $Q^*\subset \subset \Omega$ and equal to $1$ on a neighborhood of $Q$. Then $\vec U=-\frac{1}{\Delta}\vec\nabla\wedge(\psi\vec\nabla \vec u)$ is such that $\vec H=\vec U-\vec u$ is a harmonic correction of $\vec u$ on $Q$.
\end{itemize}
\end{lemma}
\begin{proof} The case of $\vec H=-\vec\nabla\varpi$ has been discussed   by Wolf  \cite{WOL}.  For $\vec U=\vec u-\vec\nabla \varpi$, we have $\vec\nabla\wedge\vec U=\vec\nabla\wedge\vec  u$ and $\Delta\vec U=\Delta\vec u$, so that
\begin{equation*}\begin{split} \partial_t\vec U -\Delta\vec U+\vec U\cdot \vec\nabla\vec U&=\partial_t\vec u-\partial_t\vec\nabla\varpi-\Delta\vec u+(\vec \nabla\wedge\vec u)\wedge(\vec u-\vec\nabla\varpi)+\vec\nabla(\frac{\vert\vec U\vert^2}2)\\&= \vec\nabla(\frac{\vert\vec U\vert^2}2-\frac{\vert\vec u\vert^2}2-p_0)
+\vec f - (\vec \nabla\wedge\vec u)\wedge \vec\nabla\varpi\end{split}\end{equation*}
We may  then decompose $(\vec \nabla\wedge\vec u)\wedge \vec\nabla\varpi\in L^2_tL^2_x(Q)$ into $\vec f_1+\vec\nabla p_1$ with $\vec f_1\in L^2_t L^2_x$ and $\text{ div }\vec f_1=0$ and $p_1\in L^2_t L^6_x(Q)$ (for instance, by extending   $(\vec \nabla\wedge\vec u)\wedge \vec\nabla\varpi$ by $0$ outside $Q$ and then using the Leray projection operator). We thus find $$P= p_0+ \frac{\vert\vec u\vert^2}2 -  \frac{\vert\vec U\vert^2}2+p_1 \text{ and }  \vec F=\vec f-\vec f_1.$$

The case of $\vec U=-\frac{1}{\Delta}\vec\nabla\wedge(\psi\vec\nabla \vec u)$ has been discussed by Chamorro, Lemari\'e-Rieusset and Mayoufi in \cite{CLM, LEM}. It is worth noticing that the pressure $P$ they obtain belongs to $L^2_t L^q_x(Q)$ for every $q<3/2$.\\

Note that, in both cases, even if $\vec f$ is assumed to be more regular, we cannot get a better regularity for $\vec F$ than $L^2_{t} L^2_{x}$, because of the contribution of $(\vec\nabla\wedge\vec u)\wedge \vec H$ to the force.
 \end{proof}
 
 An important result of Chamorro, Lemari\'e-Rieusset and Mayoufi is the following one  \cite{CLM, LEM} : 
 \begin{theorem}[\bf Energy balance]
 Let $\Omega$ be a   domain in $\mathbb{R}\times\mathbb{R}^3$,  $\vec f\in L^2_{\rm loc}(\Omega)$ a divergence-free time-dependent vector field and    $\vec u$    a local weak solution of the Navier--Stokes equations on $\Omega$ (associated to the force $\vec f$).  Let $Q$ be a spherical cylinder in $\Omega$ and $p$ the pressure associated to $\vec u$ on $Q$. Then:
 \begin{itemize}
\item[A)] The quantities
$$ M(\vec u)=   \lim_{\epsilon\rightarrow 0} \left(- \text{ \rm div }(\vert \vec u\vert^2\vec u) +2 \vec u_\epsilon\cdot\varphi_\epsilon*(\vec u\cdot\vec\nabla\vec u)\right) $$
and
$$ << \text{\rm div }(p\vec u) >> =\lim_{\epsilon\rightarrow 0}  \lim_{\eta\rightarrow 0}\text{\rm div }( (p*\xi_{\eta,\epsilon})\vec u_{\epsilon,\eta})$$ are well defined in $\mathcal{D}'(Q)$. 
\item[B)] We have the energy balance on $Q$ :
$$ \partial_t\vert\vec u\vert^2= \Delta(\vert \vec u\vert^2)-2\vert\vec\nabla\otimes\vec u\vert^2-  \text{ \rm div }(\vert \vec u\vert^2\vec u)  - 2 << \text{\rm div }(p\vec u) >>   +2\vec u\cdot\vec f-M(\vec u).$$ 
\item[C)] $M(\vec u)$ can be computed as a defect of regularity. More precisely, we have, for
 $$ A_{k,\epsilon}(\vec u)=\frac{ (u_k(t,x-y)-u_k(t,x))(\vec u(t,x-y)-\vec u(t,x)) \cdot\int \varphi_\epsilon(z) (\vec u(t,x-z)-\vec u(t,x))\, dz}\epsilon $$
 and
 $$ B_{k,\epsilon}(\vec u)= \frac{ (u_k(t,x-y)-u_k(t,x)) \vert \vec u(t,x-y)-\vec u(t,x) \vert^2  }\epsilon ,$$ 
 the identity
 \begin{equation}\label{EDR} M_\epsilon(\vec u)= \sum_{k=1}^3  \int \frac 1 {\epsilon^3} \partial_k\varphi(\frac y \epsilon)   (2 A_{k,\epsilon}(\vec u)-B_{k,\epsilon}(\vec u))\, dy-C_\epsilon(\vec u) 
 \end{equation} 
 where $\lim_{\epsilon\rightarrow 0} C_\epsilon(\vec u)=0$ in $\mathcal{D}'(Q)$.
\item[D)]  If $\vec U=\vec u+\vec H$ where $\vec H$ is a harmonic correction of $\vec u$, then $M(\vec U)=M(\vec u)$.
\end{itemize}
\end{theorem}
\begin{proof} The key tool is identity (\ref{EDR}) which has been described by Duchon and Robert \cite{DUC} for any divergence-free vector field $\vec u$ in $L^\infty_tL^2_x(Q)\cap L^2_tH^1_x(Q)$.  Let us remark that if $w_1$ and $w_2$ belong to $L^\infty_tL^2_x(Q)\cap L^2_tH^1_x(Q)$ and $w_3$ to $L^\infty_t {\rm Lip}_x(Q)$ then we have obviously
 $$ \lim_{\epsilon\rightarrow 0} \int   \frac 1 {\epsilon^3} \partial_k\varphi(\frac y \epsilon)   \frac{(w_1(x-y)-w_1(x))(w_2(x-y)-w_2(x))(w_3(x-y)-w_3(x))}\epsilon\, dy=0.$$ 
 Thus, if   $\vec H$ is a harmonic correction of $\vec u$, we have $\lim_{\epsilon\rightarrow 0} M_\epsilon(\vec u+\vec H)-M_\epsilon(\vec u)=0$. Since the limits $  \lim_{\epsilon\rightarrow 0} M_\epsilon(\vec u+\vec H)$ and $\lim_{\epsilon\rightarrow 0} (M_\epsilon(\vec u) +2 \lim_{\eta\rightarrow 0}\text{\rm div }( (p*\xi_{\eta,\epsilon})\vec u_{\epsilon,\eta}))$ are well defined in $\mathcal{D}'(Q)$, we find that $M(\vec u)$ and $<<\text{ \rm div }(p\vec u) >>$ are well defined and that $M(\vec u)=M(\vec u+\vec H)$.
 \end{proof}
 
 Of course, if $\vec u$ is regular enough, we have $M(\vec u)=0$. Due to formula  (\ref{EDR}), Duchon and Robert \cite{DUC} could see that when $\vec u$ belongs locally to $L^3_t (B^{1/3}_{3,q})_x$ with $q<+\infty$, then $M(\vec u)=0$. This is the case when the classical criterion $\vec u\in L^4_{t,x}(\Omega)$ is fulfilled, since $L^4_tL^4_x\cap L^2_t H^1_x\subset L^3_t(B^{1/3}_{3,3})_x$. In particular, the support of the distribution $M(\vec u)$ is a subset of the set $\Sigma(\vec u)$ of singular points.
   \section{Dissipativity and partial regularity.}
The best result we know about (partial) regularity of weak solutions has been given in 1982 by Caffarelli, Kohn and Nirenberg \cite{CKN, LAD}. Their result is based on the notion of suitable solutions (due to Scheffer \cite{SCH}):
 
 \begin{definition}[\bf Suitable solutions] Let $\vec u$ be a local weak solutions of the Navier--Stokes solutions on a domain $\Omega\subset \mathbb{R}\times\mathbb{R}^3$. Then $\vec u$ is suitable if if satisfies the following two conditions :
 \begin{itemize}
 \item the pressure $p$ is locally in $L^{3/2}_{t,x}$,
 \item $M(\vec u)\geq 0$ (i.e. $M(\vec u)$ is a non-negative locally finite Borel measure).
 \end{itemize}
 \end{definition}
Let us define now the parabolic metric $\rho((t,x),(s,y))=\max(\sqrt{\vert t-s\vert},\vert x-y\vert^2)$ and the parabolic cylinders $Q_r(t,x)=\{(s,y): \ \rho((t,x),(s,y))<r\}$.
 
 \begin{theorem}[\bf Caffarelli, Kohn and Nirenberg's regularity theorem]  Let $\Omega$ be a   domain in $\mathbb{R}\times\mathbb{R}^3$,  $\vec f\in L^2_{\rm loc}(\Omega)$ a divergence-free time-dependent vector field and    $\vec u$    a local weak solution of the Navier--Stokes equations on $\Omega$ (associated to the force $\vec f$). Assume that moreover
 \begin{itemize}
 \item $\vec u$ is suitable,
 \item the force $\vec f$ is regular : $\vec f$ belongs locally to $L^2_tH^1_x$,
 \end{itemize} Then:
 \begin{itemize}
 \item if $(t,x)\notin \Sigma(\vec u)$, there exists a neighborhood of $(t,x)$ on which $\vec u$ is H\"olderian (with respect to the parabolic metric $\rho$) and we have
 $$ \lim_{r\rightarrow 0} \frac{1}{r}\iint_{Q_r(t,x)} \vert \vec\nabla\otimes\vec u\vert^2\, ds\, dy =0.$$
 \item if $(t,x)\in  \Sigma(\vec u)$,then
 $$ \limsup_{r\rightarrow 0} \frac{1}{r}\iint_{Q_r(t,x)} \vert \vec\nabla\otimes\vec u\vert^2\, ds\, dy >\epsilon^*,$$ where $\epsilon^*$ is a positive constant (which doesn't depend on $\vec u$, $\vec f$ nor $\Omega$).
 \end{itemize}
 \end{theorem}
 
 The size of $\Sigma(\vec u)$ is then easily controlled with the following lemma :
 \begin{lemma}[\bf Parabolic Hausdorff dimension.]Let $u$ belongs locally to $L^2_t H^1_x$ and let $\Sigma$ be the set defined by
 $$ (t,x)\in \Sigma\Leftrightarrow  \limsup_{r\rightarrow 0} \frac{1}{r}\iint_{Q_r(t,x)} \vert \vec\nabla  u\vert^2\, ds\, dy >0.$$ Then  $\Sigma$ has parabolic one-dimensional Hausdorff measure equal to $0$.
 \end{lemma}

  Chamorro, Lemari\'e--Rieusset and Mayoufi \cite{CLM} have considered the case where no integrability assumptions were made on the pressure $p$. This implies to change the definition of suitable solutions. Following \cite{DUC}, they introduced the notion of dissipative solutions :

 \begin{definition}[\bf Dissipative solutions] Let $\vec u$ be a local weak solutions of the Navier--Stokes solutions on a domain $\Omega\subset \mathbb{R}\times\mathbb{R}^3$. Then $\vec u$ is dissipative  if   $M(\vec u)\geq 0$.
  \end{definition}
 
 A similar notion has been given by Wolf \cite{WOL}.  Indeed, if $\vec u$ is dissipative and if we use the harmonic correction $\vec H=-\vec\nabla\varpi$, we find, for $\vec U=\vec u+\vec H$ :
 \begin{eqnarray*}
 M(\vec U)&=&-\partial_t\vert\vec U\vert^2+\Delta(\vert \vec U\vert^2)-2\vert\vec\nabla\otimes\vec U\vert^2-   \text{ \rm div }(\vert \vec U\vert^2\vec U)  - 2  \text{ \rm div }(P\vec U)    +2\vec U\cdot\vec F\\ 
 &=&-\partial_t\vert\vec U\vert^2+\Delta(\vert \vec U\vert^2)-2\vert\vec\nabla\otimes\vec U\vert^2-   \text{ \rm div }((\vert \vec U\vert^2+2p_0)\vec U) \\ 
 &  &  +2\vec U\cdot\vec f-2\vec U\cdot\vec f_1-2\text{ \rm div }(p_1\vec U) \\ 
 &=&-\partial_t\vert\vec U\vert^2+\Delta(\vert \vec U\vert^2)-2\vert\vec\nabla\otimes\vec U\vert^2-   \text{ \rm div }((\vert \vec U\vert^2+2p_0)\vec U) \\
 & &    +2 \vec U\cdot\vec f +2 \vec U\cdot (\vec \nabla\varpi\wedge (\vec\nabla\wedge\vec U)).
\end{eqnarray*}
 Writing $M(\vec U)\geq 0$ is exactly expressing that $\vec u$ is a generalized suitable solution, as defined by Wolf.\\
 
 Another tool used by   Chamorro, Lemari\'e--Rieusset and Mayoufi is the notion of parabolic Morey space :
 
  \begin{definition}[\bf Parabolic Morrey spaces]  A function   $\theta$ belongs to the parabolic Morrey space $\mathcal{M}^{s,\tau}(\Omega)$ if 
$$ \sup_{x_0,t_0,r}  \frac{1}{r^{ 5(1-\frac s \tau)}}\iint_\Omega 1_{\vert t-t_0\vert<r^2} 1_{\vert x-x_0\vert<r} \vert\theta(t,x)\vert^{ s}\, dt\, dx <+\infty.$$
  \end{definition}
  
  Parabolic Morrey spaces have been used by Kukavica \cite{KUK} in a variant of Caffarelli, Kohn and Nirenberg's theorem \cite{CKN}, and by O'Leary \cite{OLE, LEM} in a variant of Serrin's regularity theorem \cite{SER} : 
\begin{theorem}[\bf Kukavica's theorem]   There exists  a positive constant $\epsilon^*$  such that  the following holds : 
If $\vec U$ is a solution of the Navier--Stokes equations on a domain $\Omega_1$, associated to a force $\vec F$ and a pressure $P$ and if $x_0$, $t_0$, $\vec U$, $P$ and $\vec F$ satisfy the following assumptions 
\begin{itemize}
\item    $\vec  U$  belongs to  $L^\infty_t L^2_x\cap L^2_t H^1_x$,
\item  $P\in L^{3/2}_{t,x}(\Omega)$,
\item  $\text{ div }\vec F=0$ and   $\vec  F \in  L^2_{t,x}(\Omega_1)$,  
\item  $\vec U$ is suitable,  
\item $(t_0,x_0)\in \Omega_1$ and  $$\limsup_{r\rightarrow 0} \frac{1}{r}\iint_{(t_0-r^2,t_0+r^2)\times B(x_0,r)} \vert\vec\nabla\otimes \vec U\vert^2\, ds\, dx< \epsilon^*,$$
\end{itemize}
then  there exists  $\tau>5$  and a neighborhood      $\Omega_2$ of $(t_0,x_0)$  such that  $  \vec U\in \mathcal{M}^{3,\tau}(\Omega_2)$. 
\end{theorem}

\begin{theorem}[\bf O'Leary's theorem]   If $\vec u$ is a solution of the Navier--Stokes equations on a domain $\Omega_2$, associated to a force $\vec f$  and if   $\vec u$ and $\vec f$ satisfy the following assumptions 
\begin{itemize}
\item    $\vec u$  belongs to  $L^\infty_t L^2_x\cap L^2_t H^1_x$, 
\item  $\text{ div }\vec f=0$ and   $\vec  f \in  L^2_t H^k_x(\Omega_2)$  for some $k\in\mathbb{N}$,
 \item $ \vec u\in \mathcal{M}^{s,\tau}(\Omega_2) \text{ with }  \tau>5\text{ and } 2<s\leq \tau$,
 \end{itemize}
 then, for every subdomain $\Omega_3$ which is relatively compact in $\Omega_2$, we have
 $$ \vec u\in L^\infty_t H^{k+1}_{x}\cap L^2_t H^{k+2}_{x}(\Omega_3).$$
\end{theorem}
 
Using those theorems,  Chamorro, Lemari\'e--Rieusset and Mayoufi \cite{CLM} could prove the following theorem (which is essentially the result proved  previously by Wolf \cite{WOL}) : 
  \begin{theorem}[\bf Wolf's theorem]  Let $\Omega$ be a   domain in $\mathbb{R}\times\mathbb{R}^3$,  $\vec f\in L^2_{\rm loc}(\Omega)$ a divergence-free time-dependent vector field and    $\vec u$    a local weak solution of the Navier--Stokes equations on $\Omega$ (associated to the force $\vec f$). Assume that moreover
 \begin{itemize}
 \item $\vec u$ is dissipative,
 \item the force $\vec f$ is regular : $\vec f$ belongs locally to $L^2_tH^1_x$,
 \end{itemize} 
 Then:
 \begin{itemize}
 \item if $(t,x)\notin \Sigma(\vec u)$, then
 $$ \lim_{r\rightarrow 0} \frac{1}{r}\iint_{Q_r(t,x)} \vert \vec\nabla\otimes\vec u\vert^2\, ds\, dy =0.$$
 \item if $(t,x)\in  \Sigma(\vec u)$, then
 $$ \limsup_{r\rightarrow 0} \frac{1}{r}\iint_{Q_r(t,x)} \vert \vec\nabla\otimes\vec u\vert^2\, ds\, dy \geq\epsilon^*$$ where $\epsilon^*$ is a positive constant (which doesn't depend on $\vec u$, $\vec f$ nor $\Omega$).
 \end{itemize}

 \end{theorem}
 
 \begin{proof} We sketch the proof given in \cite{CLM, LEM}. Let $\epsilon^*$ be the constant in Kukavica's theorem. Let $(x_0,t_0)\in \Omega$ with
  $$ \limsup_{r\rightarrow 0} \frac{1}{r}\iint_{Q_r(t_0,x_0)} \vert \vec\nabla\otimes\vec u\vert^2\, ds\, dy <\epsilon^*.$$ We introduce a harmonic correction $\vec H$  on a cylindric neighborhood of $(x_0,t_0)$ and consider the vector field $\vec U=\vec u+\vec H$. If $\vec u$ is dissipative, then $\vec U$ is suitable, associated to a force $\vec F\in L^2_tL^2_x(Q)$ and a pressure $P\in L^{3/2}_tL^{3/2}_x(Q)$. Moreover,
    $$ \limsup_{r\rightarrow 0} \frac{1}{r}\iint_{Q_r(t_0,x_0)} \vert \vec\nabla\otimes\vec U\vert^2\, ds\, dy  =\limsup_{r\rightarrow 0} \frac{1}{r}\iint_{Q_r(t_0,x_0)} \vert \vec\nabla\otimes\vec u\vert^2\, ds\, dy <\epsilon^*.$$ 
    Thus, by Kukavica's theorem,    there exists  $\tau>5$  and a neighborhood      $\Omega_2\subset Q$ of $(t_0,x_0)$  such that  $  \vec U\in \mathcal{M}^{3,\tau}(\Omega_2)$.  As $\vec u=\vec U-\vec H$, we see that we have as well $  \vec u\in \mathcal{M}^{3,\tau}(\Omega_2)$.  As $\vec f\in L^2_t H^1_x$, we may use O'Leary's theorem and find that, on a cylindric neighborhood $\Omega_3$ of $(t_0,x_0)$, we have $\vec u\in L^\infty_t H^2_x(\Omega_3)\subset L^\infty_{t,x}(\Omega_3)$. Thus, $(t_0,x_0)\notin \Sigma(\vec u)$.
 \end{proof}  
 
 \section{Weak convergence of local weak solutions.}\label{SERR}
 
 In this final section, we prove Theorem \ref{TheoCLM}. Recall that  we consider  a sequence  $(\vec f_n)_{n\in\mathbb{N}}$ of  divergence-free time-dependent vector fields on a domain  $\Omega \subset \mathbb{R}\times\mathbb{R}^3$  and   a sequence   $(\vec u_n)_{n\in\mathbb{N}}$  of   local  weak solutions of the Navier--Stokes equations on $\Omega$ (associated to the forces $\vec f_n$) such that, for each  cylinder $Q\subset\subset \Omega$,   we have
 \begin{itemize} 
 \item $\vec f_n\in L^2_t H^1_x(Q)$ and $\vec f_n$ converges weakly in $L^2_{t} H^1_{x}$ to a limit $\vec f$,
 \item the sequence $\vec u_n$ is bounded in $L^\infty_tL^2_x(Q)\cap L^2_t H^1_x(Q)$ and converges weakly in $ L^2_t H^1_x(Q)$ to  a limit $\vec u$, 
 \item for every $n$, $\vec u_n$ is bounded on $Q$ (the bound depending on $n$).
 \end{itemize}
We know that we may define a pressure $p_n$ on $Q$ and that we have the energy equality
$$ M(\vec u_n)=0,$$
 where
 \begin{eqnarray*}
M(\vec u_n)&=&- \partial_t\vert\vec u_n\vert^2+ \Delta(\vert \vec u_n\vert^2)-2\vert\vec\nabla\otimes\vec u_n\vert^2-  \text{ \rm div }(\vert \vec u_n\vert^2\vec u_n) \\
& & - 2 << \text{\rm div }(p_n\vec u_n) >>   +2\vec u_n\cdot \vec f_n .
\end{eqnarray*}
Our aim is then to prove that the limit $\vec u$ is a solution to the Navier--Stokes aquations associated to the limit $\vec f$ and that this solution is dissipative :
 $$ M(\vec u)\geq 0.$$
We cannot give a direct proof, as it is possible that no term in the definition of $M(\vec u_n)$  converge to the corresponding term in $M(\vec u)$ : $p$ is not the limit in $\mathcal{D}'$ of $p_n$ and $\vert \vec u \vert^2$ is not the limit in $\mathcal{D}'$ of $\vert\vec u_n\vert^2$\dots \ It is easy to find an example of such a bad behavior by studying Serrin's example of smooth in space and singular in time solution to the Navier--Stokes equations \cite{SER} :
 
 \begin{example}[\bf Serrin's example]  Let $\psi$ be defined on a neighborhood of $B(x_0,r_0)$ and be harmonic,
$ \Delta\psi=0$, and let $\vec f=0$ and $$\vec u=\alpha(t)\vec\nabla\psi(x),$$ where $\alpha\in L^\infty((a,b))$. Then $\vec u$ is a local weak solution of the Navier--Stokes equations on $(a,b)\times B(x_0,r_0)$ :
 $$ \partial_t \vec u= \Delta\vec u-\vec u\cdot \vec\nabla \vec u-\vec\nabla( - \dot \alpha  \psi -\frac{\vert\vec u\vert^2}2)+\vec f.$$
 \end{example}
 
 Clearly, if $\alpha$ is not regular, the pressure $p$ has no integrability in the time variable (because of the presence of the singular term $\dot \alpha(t)$)and $\vec u$ has no regularity in the time variable. Thus, $\vec u$ is dissipative (as a matter of fact, $M(\vec u)=0$) but not suitable, as it violates both assumptions and conclusions of the Caffarelli, Kohn and Nirenberg theorem.
\\

Let us adapt this example to our problem. We define $$ \vec u_n(t,x)= \cos(nt) \left(\begin{matrix} x_1 \cr -x_2\cr 0\end{matrix}\right)$$ 

\begin{itemize}\item   $\vec u_n$ is a solution on $\mathbb{R}\times \mathbb{R}^3$ of
\begin{equation*} \left\{\begin{split} \partial_t \vec u_n=&\Delta \vec u_n- (\vec u_n\cdot \vec\nabla)\vec u_n-\vec\nabla p_n  \\  {\rm div}\vec u_n=&0\end{split}\right. \end{equation*}
 \item  In this example, we have for a bounded domain $\Omega_0$  $$\vec u_n\rightharpoonup 0 $$ in $L^2_t H^1_x(\Omega_0)$ and 
 $$(\vec u_n\cdot \vec\nabla)\vec u_n\rightharpoonup \frac 1 2  \left(\begin{matrix} x_1 \cr x_2\cr 0\end{matrix}\right)\neq 0, $$ 
 in $\mathcal{D}'(\Omega_0)$.  
\end{itemize} 
In order to circumvene this problem of non-convergence, we shall use two tools~: equations on vorticities $\vec \omega_n=\vec\nabla\wedge \vec u_n$ and on harmonic corrections $\vec U_n=\vec u_n+\vec H_n=-\frac{1}{\Delta}\vec\nabla\wedge \left( \psi (\vec\nabla\wedge \vec u_n)\right)$.

\centerline{\bf Step 1 : Vorticities.}
On a cylinder $Q\subset\subset \Omega$, we may write the Navier--Stokes equations on the divergence-free vector field $\vec u_n$ in many ways. The first one is given by equation (\ref{WNS})  :  for every smooth compactly supported divergence-free vector field $\vec\phi\in \mathcal{D}(Q)$ we have
\begin{equation*} 
\iint_Q \vec u_n\cdot (\partial_t\vec\phi+\Delta\vec\phi)+ \vec u_n\cdot (\vec u_n\cdot \vec\nabla\vec\phi)+\vec f_n\cdot \vec\phi \, dt\, dx=0.
 \end{equation*}
 We may rewrite this equation as:
\begin{equation*}
 \partial_t\vec u_n = \Delta\vec u_n -  \vec u_n\cdot \vec\nabla\vec u_n+\vec  f_n\text{ in } (\mathcal{D}_\sigma(Q))'
\end{equation*}
where $\mathcal{D}_\sigma(Q)$ is the space of smooth compactly supported divergence-free vector fields on $Q$.\\

The second one is given by equations  (\ref{PRES}): for a distribution $p_n\in\mathcal{D}'(Q)$, we have
\begin{equation*} 
\partial_t\vec u_n=\Delta\vec u_n-\vec u_n\cdot \vec\nabla\vec u_n-\vec \nabla p_n+\vec f _n\quad \text{ in } \mathcal{D}'(Q).
\end{equation*}
 
The next one is based on the identity
$$ \vec u_n\cdot \vec\nabla \vec u_n= \vec\omega_n\wedge \vec u_n+\vec\nabla (\frac{\vert\vec u_n\vert^2}2)$$ from which we get
 \begin{equation*}
 \partial_t\vec u_n = \Delta\vec u_n -  \vec \omega_n \wedge\vec u_n+\vec  f_n\text{ in } (\mathcal{D}_\sigma(Q))'.
\end{equation*}

We have seen that, in some cases, we don't have the convergence of  $ (\vec u_n\cdot \vec\nabla)\vec u_n$ to $\vec u\cdot \vec\nabla \vec u$  in $\mathcal{D}'(\Omega_0)$.  But we shall prove the following lemma :
\begin{lemma}[\bf Convergence of the non-linear term] We have the following convergence results : 
 $$\vec \omega_n\wedge\vec u_n\rightharpoonup  \vec\omega\wedge\vec u\text{ in }\mathcal{D}'(Q)$$
 so that
  $$(\vec u_n\cdot \vec\nabla)\vec u_n\rightharpoonup  \vec u\cdot \vec\nabla\vec u \text{ in } (\mathcal{D}_\sigma(Q))' .$$
\end{lemma}
Thus, this lemma will prove the first half of Theorem \ref{TheoCLM}:   the limit $\vec u$ is a local weak solution on $\Omega$ of the Navier--Stokes equations associated to the force $\vec f$. The proof of the lemma is based on the following variant of the classical Rellich lemma \cite{LEM, MAY} : 
\begin{lemma}[\bf Rellich's lemma] Let  $-\infty<\sigma_1<\sigma_2<+\infty$.   Let $\Omega$ be a   domain in $\mathbb{R}\times\mathbb{R}^3$.   If a sequence of distribution $T_n$ is weakly convergent to a distribution $T$ in $(L^2_t H^{\sigma_2}_x)_{\rm loc}$ and if the sequence $(\partial_t T_n)$ is bounded in $(L^2_t H^{\sigma_1}_x)_{\rm loc}$, then $T_n$ is strongly convergent in $(L^2_t H^\sigma_x)_{\rm loc}$ for every $\sigma<\sigma_2$.
 \end{lemma}
 We apply Rellich's lemma to $\vec\omega_n$. We have
 $$ \partial_t\vec\omega_n =\Delta\vec\omega_n-\text{ \rm div }(\vec u_n\otimes\vec\omega_n-\vec\omega_n\otimes \vec u_n)-\vec\nabla\wedge \vec f_n,$$ 
 so that the sequence $(\partial_t \vec\omega_n)$ is bounded in $(L^2_t H^{\sigma_1}_x)_{\rm loc}$ for all $\sigma_1<-5/2$. Moreover, $\vec\omega_n$ is weakly convergent to $\vec \omega$ in $(L^2_t L^2_x)_{\rm loc}$. Thus,  $\vec\omega_n$ is strongly convergent in $(L^2_t H^{-1}_x)_{\rm loc}$. As $\vec u_n$ is weakly convergent
 to $\vec u$ in $(L^2_t H^1_x)_{\rm loc}$, we find that $\vec \omega_n\wedge\vec u_n$ is weakly convergent to $  \vec\omega\wedge\vec u$ in $\mathcal{D}'(\Omega)$.\\
 
 \centerline{\bf Step 2 : Harmonic corrections.}
We now end the proof of Theorem \ref{TheoCLM} by checking the dissipativity of the limit $\vec u$. We restate the theorem as a result of stability for dissipativity : 
 \begin{theorem}[\bf Dissipative limits] 
 Let $\Omega$ be a   domain in $\mathbb{R}\times\mathbb{R}^3$. Assume that we have sequences $\vec f_n$ of  divergence-free time-dependent vector fields and    $\vec u_n$  of   local  weak solutions of the Navier--Stokes equations on $\Omega$ (associated to the forces $\vec f_n$) such that, for each  cylinder $Q\subset\subset \Omega$,   we have
 \begin{itemize} 
 \item $\vec f_n\in L^2_t H^1_x(Q)$ and $\vec f_n$ converges weakly in $L^2_{t} H^1_{x}$ to a limit $\vec f$,
 \item the sequence $\vec u_n$ is bounded in $L^\infty_tL^2_x(Q)\cap L^2_t H^1_x(Q)$ and converges weakly in $ L^2_t H^1_x(Q)$ to  a limit $\vec u$, 
 \item for every $n$, $\vec u_n$ is dissipative.
 \end{itemize}
 Then the limit $\vec u$ is a  dissipative local weak solution on $\Omega$ of the Navier--Stokes equations associated to the force $\vec f$.
 \end{theorem}
 
 \begin{proof}  We already know that $\vec u $  is a  local weak solution on $\Omega$ of the Navier--Stokes equations associated to the force $\vec f$.  We have to prove its dissipativity.\\
 
  Let $Q\subset\subset \Omega$ be a cylinder and $\psi\in\mathcal{D}(\Omega)$ be a cut-off function which is equal to $1$ on a neighborhood of $Q$. In order to prove that $\vec u$ is dissipative, we shall prove that the harmonic correction $\vec U=\vec H+\vec u=-\frac 1 \Delta \vec\nabla\wedge\left(\psi (\vec\nabla\wedge\vec u)\right)$ is suitable.
  
  We define as well  $\vec U_n= -\frac 1 \Delta \vec\nabla\wedge\left(\psi (\vec\nabla\wedge\vec u_n)\right)$. The weak convergence of $\vec u_n$ in $(L^2_t H^1_x)_{\rm loc}(\Omega)$ implies the weak convergence of $\vec U_n$ to $\vec U$ in $L^2_tH^1_{x}(Q)$. Moreover, the uniform boundedness of the sequence $(\vec u_n)_{n\in \mathbb{N}}$ in $(L^2_t H^1_x\cap L^\infty_t L^2_x)_{\rm loc}(\Omega)$ and of the sequence $(\vec f_n)_{n\in \mathbb{N}}$ in $(L^2_t H^1_x)_{\rm loc}(\Omega)$ implies that the sequences of pressure $P_n$ and of forces $\vec F_n$ associated to $\vec U_n$ are uniformly bounded (respectively in $L^{3/2}_tL^{3/2}_x(Q)\cap L^2_t L^{6/5}_x(Q)$ and in $L^2_t L^2_x(Q)$). Thus, $\partial_t\vec U_n$ is bounded in $L^2_t H^{-2}_x(Q)$ and Rellich's lemma gives us that $\vec U_n$ is strongly convergent to $\vec U$ in $(L^2_t L^2_x)_{\rm loc}(Q)$ (and, since $\vec U_n$ is bounded in $L^{10/3}_t L^{10/3}_x(Q)$, we have strong convergence in $(L^3_t L^3_x)_{\rm loc}(Q)$ as well).\\
  
  Taking subsequences, we may assume that the bounded sequences $P_n$ (in \\  $L^{3/2}_tL^{3/2}_x(Q)$), $\vec F_n$ (in $L^2_t L^2_x(Q)$) and $\vert\vec\nabla U_n\vert^2$ (in $L^1_t L^1_x(Q)$) converge weakly in $\mathcal{D}'$ to limits $P_\infty\in L^{3/2}_tL^{3/2}_x(Q)$, $\vec F_\infty\in L^2_t L^2_x(Q)$ and $\nu_\infty$ (a non-negative finite measure on $Q$). In particular, we have enough convergence to see that every term in the right-hand side of equality
$$ M(\vec u_n)=- \partial_t\vert\vec U_n\vert^2+ \Delta(\vert \vec U_n\vert^2)-2\vert\vec\nabla\otimes\vec U_n\vert^2-  \text{ \rm div }(\vert \vec U_n\vert^2\vec U_n)  - 2   \text{ \rm div }(P_n\vec U_n)    +2\vec U_n\cdot \vec U_n $$ 
has a limit, so that $\nu_1=\lim_{n\rightarrow +\infty}M(\vec U_n)$ exists and
$$ \nu_1=- \partial_t\vert\vec U\vert^2+ \Delta(\vert \vec U\vert^2)-2\nu_\infty-  \text{ \rm div }(\vert \vec U\vert^2\vec U)  - 2  \text{ \rm div }(P_\infty\vec  U)     +2\vec U\cdot\vec F_\infty .$$ 

As $M(\vec U_n)\geq 0$, we find that $\nu_1\geq 0$. Moreover, by the Banach--Steinhaus theorem, we find that $\nu_2=\nu_\infty-\vert \vec\nabla\otimes \vec U\vert^2\geq 0$.  As $M(\vec U)=\nu_1+2\nu_2$, we have $M(\vec U)\geq 0$. Hence, $\vec U$ is suitable and $\vec u$ is dissipative.
 \end{proof}
 
 %%%%%%%%%%%%%%%%%%%%%%
\bibliographystyle{amsalpha}

\end{document}